\documentclass[11pt]{amsart}%
\usepackage{color}
\usepackage{amsmath}
\usepackage{amssymb}
\usepackage{amsfonts}
\usepackage[latin1]{inputenc}
\usepackage{graphicx}%
\providecommand{\U}[1]{\protect\rule{.1in}{.1in}}
\providecommand{\U}[1]{\protect\rule{.1in}{.1in}}

\DeclareMathSymbol{\subsetneqq}{\mathbin}{AMSb}{36}

\theoremstyle{plain}
\numberwithin{equation}{section}
\newtheorem{theorem}{Theorem}[section]

\newtheorem{lemma}{Lemma}[section]

\newtheorem{definition}{Definition}[section]

\begin{document}
\title[The continuous version of the Moudafi's viscosity approximation method]{On the convergence of the continuous version of the Moudafi's viscosity approximation method}
\author{Ramzi May}%
\address{Mathematics Department, College of Science, King Faisal University, P.O. 380, Ahsaa 31982, Kingdom of Saudi Arabia}
\email{rmay@kfu.edu.sa}
%\thanks{This work is supported by the Deanship of Scientific Research at King Faisal University}
%\subjclass{34D05; 34D20; 46N10; 46N20}
\keywords{Hilbert space; Nonexpansive mappings; Viscosity approximation method; Asymptotic behavior}
\date{September 27, 2023}
\maketitle
\begin{abstract}
We study the asymptotic
behavior of the trajectories of the continuous dynamical system (CDS) associated to the
the discrete viscosity approximation method for fixed point problem of nonexpansive mapping (DDS)
which was introduced by Moudafi in 2000 [A. Moudafi, Viscosity approximation methods for
fixed points problems, J. Math. Anal. Appl. 241 (2000), 46-55]. We establish that
the trajectories $x(t)$ of the system (CDS) and the sequences $(x_{n})$ generated by the the discrete process (DDS) have a very similar asymptotic behaviors.
\end{abstract}
\section{Introduction}
Let $\mathcal{H}$ be a real Hilbert space with inner product $\langle.,.\rangle$ and associated  norm $\left\Vert . \right\Vert$. Throughout this paper,  $C$ is a nonempty closed and convex
subset of $\mathcal{H}$, $f:C\rightarrow C$ is a strictly contraction mapping with
coefficient $\alpha\in\lbrack0,1[$ i.e.,
 $$\left\Vert f(x)-f(y)\right\Vert
\leq\alpha\left\Vert x-y\right\Vert \forall x,y \in C,$$ and  $T:C\rightarrow
C$ is a nonexpansive mapping i.e.,
$$\left\Vert T(x)-T(y)\right\Vert
\leq\left\Vert x-y\right\Vert \forall x,y\in C.$$
We assume moreover that the set
$Fix(T)=\{x\in C:T(x)=x\}$ of fixed points of $T$ is nonempty. We recall that Kirk \cite{Kir} and Browder \cite{Bro}  independently proved in 1965 that if in addition the set $C$ is bounded in $\mathcal{H}$ then $T$ has at least one fixed point.
\par A wide variety of problems in convex optimization and variational analysis can be reduced to problems of fixed points of nonexpansives mapping (see for instance \cite{Byr},\cite{Byr},\cite{FMP},\cite{QX}  and references therein). Hence, It is therefore of interest to construct numerical efficient algorithms, continuous or discrete, that approximate such fixed points. Let us first notice that the classical and natural iterative process $x_{n+1}=T(x_{n})$  may not converge to a fixed point of $T$. To see this, it suffices to consider the simplest example when $T=-I$, where $I$ stands for the identity operator of $\mathcal{H}$. However, for any initial data $x_0\in C$, the solution $x(.)$ of the continuous dynamical
\begin{equation}
\left\{
\begin{array}
[c]{l}%
x^{\prime}(t)+x(t)=T(x(t)),t\geq0\\
x(0)=x_{0},
\end{array}
\right.  \label{t1}%
\end{equation}
associated to the discrete process $x_{n+1}=T(x_{n})$ converges weakly in $\mathcal{H}$ as $t\rightarrow+\infty$  to a fixed point of $T$ (see \cite{Bru}). We notice that the explicit Euler discretization with variable step size $h_n=\theta_n$ of the dynamical system (\ref{t1}) leads to the classical Krasnoselskii-Mann algorithm (see \cite{Man} and \cite{Kra}):
\begin{equation}
x_{n+1}=\theta_n x_n+(1-\theta_n)T(x_n).
\label{tt}
\end{equation}
It is Well-known (see for instance \cite{Gro}) that, for any initial data $x_1\in C$, the sequence $\{x_n\}$ generated by (\ref{tt}) converges weakly in $\mathcal{H}$ to a fixed point of $T$ provided that the sequence $\{\theta_n\}$ belongs to the interval $[0,1]$ and satisfies the condition
\begin{enumerate}
  \item [(C$_0$)] $\sum_{n=0}^{\infty}(1-\theta_n)\theta_{n}=\infty$.
\end{enumerate}
\par In a the pioneer paper \cite{Hal}, Halpern introduced an algorithm that converges strongly to a particular and well-defined fixed point of a nonexpansive mapping. Precisely, he considered the particular case when $C$ is the closed unit ball of $\mathcal{H}$ and established that, for every $x_1\in C$, the sequence defined recursively by the process $x_{n+1}=(1-\theta_n)T(x_n)$ converges strongly to the element of the set $Fix(T)$ with minimum norm if  $ \theta_n=\frac{1}{n^\theta}$ with $0<\theta<1$. Later in 1977, Lions \cite{Lio} generalized and improved  Halpern's convergence result. In fact, he proved that for every anchor point $u\in C$ and any initializing data $x_1\in C$, the sequence $\{x_n\}$ generated by the process $$ x_{n+1}=\theta_n u+(1-\theta_n)T(x_n)$$ converges strongly to the closest element $u^\star$ of the set $Fix(T)$ to $u$ provided the sequence $\{\theta_n\}$ belongs to $(0,1]$ and satisfies the conditions:
\begin{enumerate}
  \item [(C$_{1}$)] $\theta_{n}\rightarrow0$ as $n\rightarrow\infty,$
  \item [(C$_{2}$)] $\sum_{n=1}^{\infty}\theta_{n}=\infty,$
  \item [(C$_{3}$)] $\frac{\left\vert \theta_{n+1}-\theta_{n}\right\vert }%
{\theta_{n}^2}\rightarrow0$ as $n\rightarrow\infty,$
\end{enumerate}
In 2000, Moudafi \cite{Mou}
introduced the called viscosity approximation method. Precisely, he considered the iterative process%
\begin{equation}
x_{n+1}=\theta_{n}f(x_{n})+(1-\theta_{n})T(x_{n}),~n\geq1, \tag{DDS}%
\end{equation}
and proved that if $(\theta_{n})_{n}$ satisfies the conditions (C$_{1}$), (C$_{2}$) and

\begin{enumerate}
  \item [(C$_{4}$)] $\frac{\left\vert \theta_{n+1}-\theta_{n}\right\vert }{\theta_{n}\theta_{n+1}}\rightarrow0$ as $n\rightarrow\infty,$
\end{enumerate}
then, for any $x_{1}$ in $C$, the sequence
$\{x_n\}$ generated by (DDS) converges strongly in $\mathcal{H}$ to the
unique solution $q^{\ast}$ of the variational problem%
\begin{equation}
\left\{
\begin{array}
[c]{l}%
q^{\ast}\in Fix(T)\\
\langle f(q^{\ast})-q^{\ast},z-q^{\ast}\rangle\leq0,\forall z\in Fix(T).
\end{array}
\right.  \tag{VP}%
\end{equation}
Later in 2004, Xu \cite{Xu} improved Moudafi's convergence result by replacing the
condition (C$_{4}$) by the weaker one:

\begin{enumerate}
  \item [(C$_{5}$)] $\frac{\left\vert \theta_{n+1}-\theta_{n}\right\vert }%
{\theta_{n}}\rightarrow0$ as $n\rightarrow\infty$ or $\sum_{n=1}^{\infty
}\left\vert \theta_{n+1}-\theta_{n}\right\vert <\infty.$
\end{enumerate}
In our present work, we consider the continuous dynamical system%
\begin{equation}
\left\{
\begin{array}
[c]{l}%
x^{\prime}(t)+x(t)=\theta(t)f(x(t))+(1-\theta(t))T(x(t)),t\geq0\\
x(0)=x_{0},
\end{array}
\right.  \tag{CDS}%
\end{equation}
associated to the discrete algorithm (DDS) where $\theta:[0,\infty)\rightarrow (0,1]$ is a regular function. We prove that, if the
function $\theta(.)$ satisfies the continuous version of the discreet conditions
(C$_{1}$), (C$_{2}$) and (C$_{5}$), then for any initial data $x_{0}\in C$, the
system (CDS) has a unique global solution $x\in C^{1}([0,\infty),\mathcal{H})$ which
converges strongly in $\mathcal{H}$ as $t\rightarrow+\infty$ to the unique
solution $q^{\ast}$ of the problem (VP). Moreover, in the particular case
where $\theta(t)=\frac{K}{(1+t)^{\nu}}$ with $K>0$ and $\nu\in(0,1],$ we
establish an estimate on the rate of convergence of $\left\Vert
T(x(t))-x(t)\right\Vert $ as $t\rightarrow+\infty.$ Such result can be
considered as the continuous version of a recently result, established by
Lieder \cite{lie} , on the rate of convergence of the sequence $\{x_{n}%
-T(x_{n})\}$ for the process (DDS) in the case when the function $f$ is
constant and the sequence $\{\theta_{n}\}$ is given by $\theta_{n}=\frac{1}{n+2}.$
\par The rest of the paper is organized as follows. In the next section, we recall some classical notions and results from functional analysis and convex analysis that are useful in the sequel of the paper. In the third section, we study the strong convergence of the trajectory of the system (CDS) under the continuous version of the discrete conditions (C$_{1}$), (C$_{2}$) and (C$_{5}$). In the last section, we investigate the stability of the system (CDS) under relatively small perturbations and we establish an estimation on the rate of the convergence of the residual term $x(t)-T(x(t))$ in the particular case when $\theta(t)=\frac{K}{(1+t)^{\nu}}$ with $K>0$ and $0<\nu\leq1$.
\section{Preliminaries}
In this section, we recall some classical definitions and results from convex
and functional analysis and derive some simple lemmas that will be needed in proving the
main results of this paper.

We first recall the definition and the main properties of the metric projection
onto a nonempty, closed and convex subset of the Hilbert space $\mathcal{H}$.

\begin{lemma}[{\protect\cite[Proposition 1.37]{Pey}}]
\label{pro}Let $K$ be a nonempty, closed and convex subset of $\mathcal{H}$. Then the following assertions hold:
\begin{enumerate}
\item[(1)] For every $%
x\in \mathcal{H},$ there exists a unique $P_{K}(x)\in K$ such that
\begin{equation*}
\left\Vert x-P_{K}(x)\right\Vert \leq \left\Vert x-y\right\Vert \ ~\forall
y\in K.
\end{equation*}%
The operator $P_{K}:\mathcal{H}\rightarrow K$ is called the metric
projection onto $K$.
\item[(2)] For every $x\in \mathcal{H},$ $P_{K}(x)$ is the unique element of
$K$ satisfying%
\begin{equation}
\langle P_{K}(x)-x,P_{K}(x)-y\rangle\leq0,\text{ for every }y\in K.
\end{equation}
\item[(3)] The operator $P_{K}:\mathcal{H}\rightarrow K$ is
nonexpansive i.e.,
\begin{equation}
\left\Vert P_{K}(x)-P_{K}(y)\right\Vert \leq\left\Vert x-y\right\Vert ,\text{
for all }x,y\in\mathcal{H}.
\end{equation}
\end{enumerate}
\end{lemma}

The second result is a classical property of the set of fixed points of a
nonexpansive mapping

\begin{lemma}[{\protect\cite[Proposition 4.13]{BC}}] \label{fix} Let $K$ be a closed convex and nonempty
subset of $\mathcal{H}$. If $A:K\rightarrow K$ is a nonexpansive mapping then
$F_{ix}(A)=\{x\in K:A(x)=x\}$ is a closed and convex subset of $\mathcal{H}$.
\end{lemma}

The next result is a particular case of the general demi-closedness
property for nonexpansive mappings
\begin{lemma}[{\protect\cite[Corollary 4.18]{BC}}]
\label{clo}Let $K$ be a closed convex and nonempty subset of $\mathcal{H}$,
 $A:K\rightarrow K$ a nonexpansive mapping, and $\{x_n\}$  a sequence
in $K$. If $\{x_n\}$ converges weakly in $\mathcal{H}$ to some element $\bar{x}$ and $\{x_{n}-A(x_{n})\}$
converges strongly to $0$ in $\mathcal{H}$, then $x\in F_{ix}(A)$.
\end{lemma}

We will now prove a variant of Gronwall's inequality that will be used frequently in the sequel.

\begin{lemma}
\label{gro}Let $u,v,w:[0,\infty)\rightarrow[0,\infty)$ be
three continuous functions. If the function $u$ is
absolutely continuous and satisfies, for almost every $t\geq0,$ the
differential inequality
\[
u^{\prime}(t)+2v(t)u(t)\leq2w(t)\sqrt{u(t)}.
\]
Then, for very $t\geq0,$
\begin{equation}
\sqrt{u(t)}\leq e^{-V(t)}\sqrt{u(0)}+e^{-V(t)}\int_{0}^{t}e^{V(s)}w(s)ds,
\label{g1}%
\end{equation}
where $V(t)=\int_{0}^{t}v(\tau)d\tau.$
\end{lemma}

\begin{proof}
Let $\varepsilon>0.$ It is clear that the function $u_{\varepsilon}$ defined on $[0,\infty)$  by $u_{\varepsilon}(t)=\sqrt{u(t)+\varepsilon}$  is absolutely
continuous and satisfies the estimations
\begin{align*}
u_{\varepsilon}^{\prime}(t)  &  =\frac{u^{\prime}(t)}{2\sqrt{u(t)+\varepsilon
}}\\
&  \leq-v(t)\frac{u(t)}{\sqrt{u(t)+\varepsilon}}+w(t)\frac{\sqrt{u(t)}}%
{\sqrt{u(t)+\varepsilon}}\\
&  \leq-v(t)u_{\varepsilon}(t)+v(t)\frac{\varepsilon}{\sqrt{u(t)+\varepsilon}%
}+w(t)\\
&  \leq-v(t)u_{\varepsilon}(t)+w(t)+v(t)\sqrt{\varepsilon}.
\end{align*}
Therefore, for almost every $t\geq0,$
\[
\left(  e^{V(t)}u_{\varepsilon}(t)\right)  ^{\prime}\leq e^{V(t)}%
w(t)+\sqrt{\varepsilon}\left(  e^{V(t)}\right)  ^{\prime}.
\]
Integrating the later differential inequality, we obtain%
\[
\sqrt{u(t)+\varepsilon}\leq e^{-V(t)}\sqrt{u(0)+\varepsilon}+e^{-V(t)}\int
_{0}^{t}e^{V(s)}w(s)ds+\sqrt{\varepsilon}(1-e^{-V(t)}),~\forall t\geq0.
\]
Hence, by letting $\varepsilon\rightarrow0,$ we get the desired inequality
(\ref{g1}).
\end{proof}
We close this section by proving the following simple result that will be useful in the proof of the existence of the solutions $x(.)$ of the differential system (CDS).
\begin{lemma}
\label{int} Let $K$ be a nonempty closed and convex subset of the Hilbert space $%
\mathcal{H}$. Let $a<b$ be two reals numbers, $w:[a,b]\rightarrow \lbrack 0,\infty
)$ a continuous function such that $\int_{a}^{b}w(t)dt>0$, and $%
v:[a,b]\rightarrow \mathcal{H}$ a continuous functions such that $v(t)\in K$
for every $t\in [a,b].$ Then
\[
I:=\frac{1}{\int_{a}^{b}w(t)dt}\int_{a}^{b}v(s)w(s)ds\in K.
\]
\end{lemma}

\begin{proof}
For every $N\in \mathbb{N},$ set $I_{N}=\sum_{k=0}^{N-1}v(x_{k,N})\alpha
_{k,N}$ with $x_{k,N}=a+k\frac{b-a}{N}$ and $\alpha _{k,N}=\frac{1}{%
\int_{a}^{b}w(t)dt}\int_{x_{K,N}}^{x_{k+1,N}}w(s)ds.$ From the convexity of the set $%
K,$ it follows that $I_{N}\in K$ for every $N\in \mathbb{N}$. On the other
hand, the uniform continuity of the function $v$ implies directly that the sequence $\{I_{N}\}$ converges strongly
in $\mathcal{H}$ to $I.$ Hence, by using the fact that $K$ is a closed
subset of $\mathcal{H}$, we conclude that $I\in K.$
\end{proof}
\section{Strong convergence of the trajectories of the dynamical system(CDS)}

In this section, we study the asymptotic behavior of the solution to the dynamical system%
\begin{equation}
\left\{
\begin{array}
[c]{l}%
x^{\prime}(t)+x(t)=\theta(t)f(x(t))+(1-\theta(t))T(x(t)),t\geq0\\
x(0)=x_{0},
\end{array}
\right.  \tag{CDS}%
\end{equation}
where $\theta:[0,\infty)\rightarrow (0,1]$ is an absolutely continuous function and $x_0\in C$ is a given initial data.

Before stating our main result, let us first precise the notion of a trajectory of the system (CDS).

\begin{definition} \label{def}
A trajectory of the system (CDS) is a continuously differentiable function
$x:[0,\infty\lbrack\rightarrow \mathcal{H}$ that satisfies the following properties:
\begin{enumerate}
\item[(1)] $x(t)\in C$ for every $t\geq0,$

\item[(2)] $x(0)=x_{0},$

\item[(3)] $x^{\prime}(t)+x(t)=\theta(t)f(x(t))+(1-\theta(t))T(x(t))$ for
every $t\geq0.$
\end{enumerate}
\end{definition}
We now state and prove the main result of the paper.
\begin{theorem}
\label{the1}The system (CDS) has a unique trajectory $x(.).$ Moreover, if in addition
the function $\theta(.)$ satisfies the following conditions

\begin{enumerate}
\item[(C'$_{1}$)] $\theta(t)\rightarrow0$ as $t\rightarrow\infty,$

\item[(C'$_{2}$)] $\int_{0}^{+\infty}\theta(t)dt=\infty,$

\item[(C'$_{5}$)] $\int_{0}^{+\infty}\left\vert \theta^{\prime}(t)\right\vert
dt<\infty$ or $\frac{\theta^{\prime}(t)}{\theta(t)}\rightarrow0$ as
$t\rightarrow\infty,$
\end{enumerate}
then $x(t)$ converges strongly in $\mathcal{H}$ as $t\rightarrow\infty$ to $q^{\ast}$
the unique solution of the variational problem%
\begin{equation}
\left\{
\begin{array}
[c]{l}%
q^{\ast}\in Fix(T)\\
\langle f(q^{\ast})-q^{\ast},z-q^{\ast}\rangle\leq0,\forall z\in Fix(T).
\end{array}
\right.  \tag{VP}%
\end{equation}
\end{theorem}
\begin{proof}
Let us first recall, for the convenience of the readers, the classical proof of the
existence and uniqueness of the solution $q^{\ast}$ of the problem (VP). First,  from
Lemma \ref{fix}, $Fix(T)$ is a closed and convex nonempty subset of $\mathcal{H}$, then the projection operator $P_{Fix(T)}$ is well-defined. Moreover, from
the second assertion of Lemma \ref{pro}, the problem (VP) is equivalent to
$q^{\ast}$ is a fixed point of the mapping $P_{Fix(T)}\circ
f:Fix(T)\rightarrow Fix(T).$ Now, since $P_{Fix(T)}$ is nonexpansive, the
mapping $P_{Fix(T)}\circ f$ is a strict contraction with coefficient $\alpha$
and therefore, according to the classical theorem of Banach,  it has a unique fixed
point. This proves the existence and the uniqueness of the solution $q^{\ast}$
of (VP).
\par We divide the second part of the proof into many steps.
\par\noindent \textbf{The first step:} We prove here the existence and the uniqueness of the
trajectory $x(t)$ of the system (CDS). To do this, we consider the Cauchy problem%
\begin{equation}
\left\{
\begin{array}
[c]{l}%
x^{\prime}(t)+x(t)=g(t,x(t)),t\geq0\\
x(0)=x_{0},
\end{array}
\right.  \label{cau}%
\end{equation}
where $g:[0,\infty)\times \mathcal{H}\rightarrow \mathcal{H}$ is the mapping defined by
$$g(t,x)=\theta(t)f(P_{C}(x))+(1-\theta(t))T(P_{C}(x)).$$
It is clear that the
function $g$ is continuous and satisfies , for every $t\geq0$ and
$x_{1},x_{2}\in \mathcal{H},$ the estimations%
\begin{align}
\left\Vert g(t,x_{1})-g(t,x_{2})\right\Vert  &  \leq(1-\gamma\theta
(t))\left\Vert x_{1}-x_{2}\right\Vert \label{est}\\
&  \leq\left\Vert x_{1}-x_{2}\right\Vert ,\nonumber
\end{align}
with $$\gamma=1-\alpha.$$ Hence, according to the classical theorem of Cauchy-Lipschitz,
system (\ref{cau}) has a unique global solution $x\in C^{1}%
([0,\infty),\mathcal{H}).$ Let $t>0$ be a fixed real. From (\ref{cau}),
\begin{align*}
x(t)  &  =e^{-t}x_{0}+e^{-t}\int_{0}^{t}e^{s}g(s,x(s))ds\\
&  =e^{-t}x_{0}+(1-e^{-t})\int_{0}^{t}g(s,x(s))w_t(s)d(s)
\end{align*}
where $w_t(s)=1_{[0,t]}(s)\frac{e^{s}}{1-e^{-t}}ds$. Since, for every $s\in[
0,t],g(s,x(s))\in C$, Lemma \ref{int} ensures that $\int_{0}^{t}g(s,x(s))w_t(s)d(s)\in C$, which implies, thanks again to the convexity of $C,$ that $x(t)\in C.$
This proves that $x(.)$ is a trajectory of the system (CDS) in the sense of the
definition \ref{def}. The uniqueness of the the trajectory of (CDS) follows
from the uniqueness of the solution of the Cauchy problem (\ref{cau}) and the
trivial fact that every trajectory of (CDS) is also a solution to (\ref{cau}).
\par\noindent \textbf{The second step:} In this step, we will prove that the trajectory $x(.)$ is bounded i.e.,
$x\in L^{\infty}([0,\infty),\mathcal{H}).$ To this end, we consider the function $u$ defined on $[0,\infty)$ by $u(t)=\left\Vert x(t)-q^{\ast
}\right\Vert ^{2}.$ Using the fact that $x(.)$ is a solution of (\ref{cau})
and the estimation (\ref{est}), we easily obtain the following estimations
\begin{align}
u(t) &  =2\langle x^{\prime}(t),x(t)-q^{\ast}\rangle \nonumber\\
&  =2\langle-x(t)+g(t,x(t)),x(t)-q^{\ast}\rangle \nonumber\\
&  =-2u(t)+2\langle g(t,q^{\ast})-q^{\ast},x(t)-q^{\ast}\rangle+2\langle
g(t,x(t))-g(t,q^{\ast}),x(t)-q^{\ast}\rangle \nonumber\\
&  =-2u(t)+2\theta(t)\langle f(q^{\ast})-q^{\ast},x(t)-q^{\ast}\rangle
+2\langle g(t,x(t))-g(t,q^{\ast}),x(t)-q^{\ast}\rangle \nonumber\\
&  \leq-2u(t)+2\theta(t)\langle f(q^{\ast})-q^{\ast},x(t)-q^{\ast}%
\rangle+2(1-\gamma\theta(t))\left\Vert x(t)-q^{\ast}\right\Vert ^{2}\label{ess}\\
&  \leq-2\gamma\theta(t)u(t)+2\theta(t)\left\Vert f(q^{\ast})-q^{\ast
}\right\Vert \left\Vert x(t)-q^{\ast}\right\Vert \nonumber \\
&  =-2\gamma\theta(t)u(t)+2\theta(t)\left\Vert f(q^{\ast})-q^{\ast}\right\Vert
\sqrt{u(t)}. \nonumber
\end{align}
Hence, by applying Lemma \ref{gro}, we get%
\[
\sqrt{u(t)}\leq e^{-\gamma\Theta(t)}\sqrt{u(0)}+\frac{\left\Vert f(q^{\ast
})-q^{\ast}\right\Vert }{\gamma}(1-e^{-\gamma\Theta(t)}),
\]
where%
\begin{equation}\label{TH}
\Theta(t)=\int_{0}^{t}\theta(s)ds.
\end{equation}
We thus conclude that
\[
\sup_{t\geq0}\left\Vert x(t)-q^{\ast}\right\Vert \leq\max\left(  \left\Vert
x_{0}-q^{\ast}\right\Vert ,\frac{\left\Vert f(q^{\ast})-q^{\ast}\right\Vert
}{\gamma}\right)  .
\]
\par\noindent \textbf{The third step:} We prove here that $x^{\prime}(t)$ converges strongly in $\mathcal{H}$ to $0$ as
$t\rightarrow\infty.$ The main idea of the proof is inspired by the proof of \cite[Lemma 8]{CPS}.
Let $\delta>0.$ We define the function $\omega$ on $[0,\infty)$ by
$\omega(t)=\left\Vert x(t+\delta)-x(t)\right\Vert ^{2}.$ Clearly,
\begin{align}
\omega^{\prime}(t)  &  =2\langle x^{\prime}(t+\delta)-x^{\prime}(t),x(t+\delta)-x(t)\rangle
\nonumber\\
&  =-2\omega(t)+2\langle g(t+\delta,x(t+\delta))-g(t,x(t)),x(t+\delta
)-x(t)\rangle\nonumber\\
&  =-2\omega(t)+2\langle g(t,x(t+\delta))-g(t,x(t)),x(t+\delta)-x(t)\rangle
\nonumber\\
&  +2\langle g(t+\delta,x(t+\delta))-g(t,x(t+\delta)),x(t+\delta
)-x(t)\rangle\nonumber\\
&  \leq-2\gamma\theta(t)\omega(t)+2\left\Vert g(t+\delta,x(t+\delta
))-g(t,x(t+\delta))\right\Vert \sqrt{\omega(t)}\nonumber\\
&  \leq-2\gamma\theta(t)\omega(t)+2M\left\vert \theta(t+\delta)-\theta
(t)\right\vert \sqrt{\omega(t)}, \label{der}%
\end{align}
where
\begin{equation}
M:=\sup_{s\geq0}\left(  \left\Vert f(x(s))\right\Vert +\left\Vert
T(x(s)\right\Vert \right)  . \label{bou}%
\end{equation}
We notice here that $M$ is finite since $x(.)\in L^{\infty}([0,\infty),\mathcal{H})$ and $f$ and $T$ are Lipschitiz continuous functions. Applying now
Lemma \ref{gro} to the inequality (\ref{der}), we deduce that for every
$t\geq0,$%
\[
\left\Vert x(t+\delta)-x(t)\right\Vert \leq e^{-\gamma\Theta(t)}\left\Vert
x(\delta)-x(0)\right\Vert +Me^{-\gamma\Theta(t)}\int_{0}^{t}e^{\gamma
\Theta(s)}\left\vert \theta(s+\delta)-\theta(s)\right\vert ds
\]
where $\Theta$ is the function defined by (\ref{TH}).
\par\noindent Dividing the last
inequality by $\delta$ and letting $\delta\rightarrow0,$ we obtain%

\begin{equation}
\left\Vert x^{\prime}(t)\right\Vert \leq e^{-\gamma\Theta(t)}\left\Vert
x^{\prime}(0)\right\Vert +Me^{-\gamma\Theta(t)}\int_{0}^{t}e^{\gamma\Theta
(s)}\left\vert \theta^{\prime}(s)\right\vert ds \forall t\geq 0.\label{rat1}%
\end{equation}
From the condition (C'$_{2}$), $\Theta(t)\rightarrow\infty$ as $t\rightarrow
\infty,$ hence in order to prove that $\left\Vert x^{\prime}(t)\right\Vert
\rightarrow0$ as $t\rightarrow\infty$ it suffices to prove that
$r(t):=e^{-\gamma\Theta(t)}\int_{0}^{t}e^{\gamma\Theta(s)}\left\vert
\theta^{\prime}(s)\right\vert ds\rightarrow0$ as $t\rightarrow\infty.$ Here we make use of the condition (C'$_{5}$). We therefore consider the two following cases:

\par\noindent \textbf{The case when $\int_{0}^{\infty}\left\vert \theta^{\prime}(s)\right\vert
ds<\infty$.}
\par\noindent Let $A>0.$ Since the function $\Theta$  is increasing then for
every $t\geq A$
\[
r(t)\leq e^{-\gamma\Theta(t)}\int_{0}^{A}e^{\gamma\Theta(s)}\left\vert
\theta^{\prime}(s)\right\vert ds+\int_{A}^{t}\left\vert \theta^{\prime
}(s)\right\vert ds.
\]
This inequality clearly implies%
\[
\overline{\lim_{t\rightarrow\infty}}~r(t)\leq\int_{A}^{\infty}\left\vert
\theta^{\prime}(s)\right\vert ds.
\]
Therefore, by letting $A\rightarrow\infty,$ we get the desired result
$\lim_{t\rightarrow\infty}r(t)=0.$

\noindent \textbf{The case when $\lim_{t\rightarrow\infty}\frac{\left\vert \theta^{\prime
}(t)\right\vert }{\theta(t)}=0$.}
\par\noindent Let $A>0.$ for every $t\geq A,$%
\begin{align}
r(t)  &  \leq e^{-\gamma\Theta(t)}\int_{0}^{A}e^{\gamma\Theta(s)}\left\vert
\theta^{\prime}(s)\right\vert ds+e^{-\gamma\Theta(t)}\int_{A}^{t}%
e^{\gamma\Theta(s)}\theta(s)ds~\sup_{s\geq A}\frac{\left\vert \theta^{\prime
}(s)\right\vert }{\theta(s)}\nonumber\\
&  \leq e^{-\gamma\Theta(t)}\int_{0}^{A}e^{\gamma\Theta(s)}\left\vert
\theta^{\prime}(s)\right\vert ds+\frac{1}{\gamma}(1-e^{-\gamma(\Theta
(t)-\Theta(A)})\sup_{s\geq A}\frac{\left\vert \theta^{\prime}(s)\right\vert
}{\theta(s)}\nonumber\\
&  \leq e^{-\gamma\Theta(t)}\int_{0}^{A}e^{\gamma\Theta(s)}\left\vert
\theta^{\prime}(s)\right\vert ds+\frac{1}{\gamma}\sup_{s\geq A}\frac
{\left\vert \theta^{\prime}(s)\right\vert }{\theta(s)}. \label{esti}%
\end{align}
Hence by letting $t\rightarrow\infty$ we get
$$\lim_{t\rightarrow\infty}r(t)\leq \frac{1}{\gamma}\sup_{s\geq A}\frac
{\left\vert \theta^{\prime}(s)\right\vert }{\theta(s)}.$$
Thus we  can conclude by letting $A$ go to $\infty$.

\par\noindent \textbf{The fourth step}: We will show that
$$
r^{\ast}:=\overline{\lim_{t\rightarrow\infty}}\langle f(q^{\ast})-q^{\ast
},x(t)-q^{\ast}\rangle\leq0.
\label{a1}$$
Since $x(.)\in L^{\infty}([0,\infty),\mathcal{H}),$ there exist
$x_{\infty}\in \mathcal{H}$ and a sequence of positive real numbers $\{t_{n}\}$ which tend
to $\infty$ such that $(\{x(t_{n})\}$ converges weakly in $\mathcal{H}$ to $x_{\infty}$
and
\begin{align}
r^{\ast}& =\lim_{n\rightarrow\infty}\langle f(q^{\ast})-q^{\ast},x(t_n)-q^{\ast}\rangle \\
&=\langle f(q^{\ast})-q^{\ast},x(\infty)-q^{\ast}\rangle.
\end{align}
On the other hand, since $x(.)$ is a trajectory of (CDS),
\begin{align}
\left\Vert x(t)-T(x(t))\right\Vert  &  \leq\theta(t)\left(
\left\Vert f(x(t))\right\Vert +\left\Vert T(x(t))\right\Vert \right)
+\left\Vert x^{\prime}(t)\right\Vert \nonumber\\
&  \leq2M~\theta(t)+\left\Vert x^{\prime}(t)\right\Vert ,\label{rat2}%
\end{align}
for every
$t\geq0$, where the constant $M$ is given by (\ref{bou}). Hence the condition (C'$_{1})$ combined with the fact that $\left\Vert x^{\prime}(t)\right\Vert\rightarrow 0$ as $t\rightarrow\infty$  implies that the sequence $\{x(t_{n}%
)-T(x(t_{n})\}$ converges strongly in $\mathcal{H}$ to $0.$ Therefore, by invoking
Lemma \ref{clo}, we deduce that $x_{\infty}\in Fix(T)$ which in turn implies that
\[
r^{\ast}=\langle f(q^{\ast})-q^{\ast},x_{\infty}-q^{\ast}\rangle\leq0.
\]
\par\noindent \textbf{The fifth step}: Finally, we establish that $x(t)\rightarrow q^{\ast}$ strongly in $\mathcal{H}$
as $t\rightarrow\infty.$
\par\noindent Define $u(t)=\left\Vert x(t)-q^{\ast}\right\Vert ^{2}.$
From (\ref{ess}), we have for every $t\geq0,$%
\[
u^{\prime}(t)+2\gamma\theta(t)u(t)\leq2\theta(t)w(t),
\]
where $$w(t):=\max\{\langle f(q^{\ast})-q^{\ast},x(t)-q^{\ast}\rangle,0\}.$$
By integrating the last differential inequality, we get%
\[
u(t)\leq e^{-2\gamma\Theta(t)}u(0)+2e^{-2\gamma\Theta(t)}\int_{0}^{t}%
e^{2\gamma\Theta(s)}\theta(s)w(s)ds,t\geq0.
\]
From the previous step,  $w(t)\rightarrow0$ as
$t\rightarrow\infty.$ Hence, by following the procedure yielding to the
estimation (\ref{esti}), we easily get $e^{-2\gamma\Theta(t)}\int_{0}^{t}%
e^{2\gamma\Theta(s)}\theta(s)w(s)ds$ $\rightarrow0$ as $t\rightarrow\infty.$ We
therefore conclude that $u(t)\rightarrow0$ as $t\rightarrow
\infty.$ This ends the proof.
\end{proof}
\section{Stability and rate of convergence of the trajectory of the dynamical system (CDS)}
In the first part of this section, we prove that the dynamical system (CDS) is stable under
the effect of a relatively small perturbation. Precisely, we prove the
following result.

\begin{theorem}\label{the2}
Let $h:[0,\infty)\rightarrow \mathcal{H}$ be a continuous function such that $h\in
L^{1}([0,\infty),\mathcal{H})$ or $\frac{\left\Vert h(t)\right\Vert }{\theta
(t)}\rightarrow0$ as $t\rightarrow\infty.$ Then for every initial data
$x_{0}\in C,$ the perturbed dynamical system%
\begin{equation}
\left\{
\begin{array}
[c]{l}%
x^{\prime}(t)+x(t)=P_{C}(\theta(t)f(x(t))+(1-\theta(t))T(x(t))+h(t)),t\geq0\\
x(0)=x_{0}%
\end{array}
\right.  \tag{PCDS}%
\end{equation}
has a unique solution $y\in C^{1}([0,\infty),H)$ such that $y(t)\in C$
for every $t\geq0.$ Moreover, $y(t)$ converges strongly in $\mathcal{H}$ as
$t\rightarrow\infty$ to $q^{\ast}$ the unique solution to the variational
problem (VP).
\end{theorem}

\begin{proof}
By proceeding exactly as in the second step of the proof of Theorem \ref{the1}, we
can establish that (PCDS) has a unique unique solution $y\in C^{1}%
([0,\infty),\mathcal{H})$ which verifies $y(t)\in C$ for every $t\geq0.$ Let now
$x(.)$ be the unique trajectory of (CDS). We consider the function $v$ defined on
$[0,\infty)$  by $v(t)=\left\Vert x(t)-y(t)\right\Vert ^{2}.$ For every
$t\geq0,$%
\begin{align}
v^{\prime}(t) &  =2\langle x^{\prime}(t)-y^{\prime}(t),x(t)-y(t)\rangle
\nonumber\\
&  =-2v(t)+2\langle\xi(t),x(t)-y(t)\rangle\label{e2}%
\end{align}
with
$$\xi(t)=P_{C}(\theta(t)f(y(t))+(1-\theta(t))T(y(t))+h(t))-(\theta
(t)f(x(t))+(1-\theta(t))T(x(t))).$$
Using the facts that $\delta(t):=\theta(t)f(x(t))+(1-\theta(t))T(x(t))\in C$ (which implies that $P_C(\delta(t))=\delta(t)$) and $P_{C}$
is nonexpansive, we easily get%
$$ \left\Vert \xi(t)\right\Vert \leq(1-\gamma\theta(t))\left\Vert
x(t)-y(t)\right\Vert +\left\Vert h(t)\right\Vert .$$
Hence, by combining this last estimate with (\ref{e2}) and using the Cauchy-Schwarz inequality, we obtain%
\[
v^{\prime}(t)\leq2\gamma\theta(t)v(t)+2\left\Vert h(t)\right\Vert \sqrt
{v(t)},~\forall t\geq0.
\]
Therefore, by applying Lemma \ref{gro}, we deduce that, for every $t\geq0,$%
\[
v(t)\leq e^{-\gamma\Theta(t)}v(0)+e^{-\gamma\Theta(t)}\int_{0}^{t}%
e^{\gamma\Theta(s)}\left\Vert h(s)\right\Vert ds,
\]
which implies, as in the proof of Theorem \ref{the1}, that $v(t)\rightarrow0$ as
$t\rightarrow\infty.$ Recalling finally that, from Theorem \ref{the1},
$x(t)\rightarrow q^{\ast}$ strongly in $\mathcal{H}$ as $t\rightarrow\infty$, we
conclude that $y(t)$ converges strongly in $\mathcal{H}$ as well to the same limit $q^{\ast}$
as $t\rightarrow\infty.$
\end{proof}

We will now establish an estimation on the rate of convergence of $x(t)-T(x(t))$ to
$0$ in the particular case when $\theta(t)=\frac{K}{(1+t)^{\nu}}$ where $K>0$
and $0<\nu\leq1.$

\begin{theorem}\label{the3}
Assume that $\theta(t)=\frac{K}{(1+t)^{\nu}}$ with $0<\nu<1$ and $K>0$ or
$\nu=1$ and $K>\frac{1}{1-\alpha}.$ Let $x(.)$ be the trajectory of the
dynamical system (CDS). Then there exists a constant $C_{\nu}>0$ such that for
every $t\geq0$%
\begin{equation}
\left\Vert x(t)-T(x(t))\right\Vert \leq\frac{C_{\nu}}{(1+t)^{\nu}}.\label{rat}%
\end{equation}

\end{theorem}

\begin{proof}
In the proof of Theorem \ref{the1} (see (\ref{rat1}) and (\ref{rat2})), we have
established that there exists a constant $M>0$ such that for every $t\geq0,$%
\begin{equation}
\left\Vert x(t)-T(x(t))\right\Vert \leq M~\theta(t)+\left\Vert x^{\prime
}(t)\right\Vert \label{eq1}%
\end{equation}
and
\begin{equation}
\left\Vert x^{\prime}(t)\right\Vert \leq e^{-\gamma\Theta(t)}\left\Vert
x^{\prime}(0)\right\Vert +Me^{-\gamma\Theta(t)}\int_{0}^{t}e^{\gamma\Theta
(s)}\left\vert \theta^{\prime}(s)\right\vert ds,\label{eq2}%
\end{equation}
where
\[
\gamma=1-\alpha,
\]
and
\[
\Theta(t)=\int_{0}^{t}\theta(s)ds.
\]
Let us now estimate the vanishing rate of the key term $$r(t):=e^{-\gamma
\Theta(t)}\int_{0}^{t}e^{\gamma\Theta(s)}\left\vert \theta^{\prime
}(s)\right\vert ds.$$
To do this, we distinguish the two cases.

\par\noindent \textbf{The first case: $\theta(t)=\frac{K}{1+t}$ with $K>\frac{1}{\gamma}.$}
\par\noindent For every
$t\geq0,$
\begin{equation}
e^{-\gamma\Theta(t)}=\frac{1}{(1+t)^{K\gamma}},\label{eq3}%
\end{equation}%
\begin{equation}
r(t)=\frac{K}{(1+t)^{K\gamma}}\int_{0}^{t}\frac{ds}{(1+s)^{2-K\gamma}}%
\leq\frac{K}{K\gamma-1}\frac{1}{1+t}\label{eq4}%
\end{equation}
\par\noindent \textbf{The second case: $\theta(t)=\frac{K}{(1+t)^{\nu}}$ with $0<\nu<1$ and $K>0.$}
\par\noindent In this case, for every $t\leq 0,$%
\begin{equation}
e^{-\gamma\Theta(t)}=e^{-\kappa(1+t)^{1-\nu}},\label{eq5}%
\end{equation}%
\[
r(t)=K\nu e^{-\kappa(1+t)^{1-\nu}}\int_{0}^{t}\frac{e^{\kappa (1+s)^{1-\nu}%
}ds}{(1+s)^{1+\nu}},
\]
with $\kappa=\frac{\gamma K}{\nu}.$ Now a simple application of the L'Hospital's rule gives
\begin{equation}
\lim_{t\rightarrow\infty}\frac{\int_{0}^{t}\frac{e^{\kappa(1+s)^{1-\nu}}}{(1+s)^{1+\nu}}ds}{\frac{1}{\kappa(1-\nu)}\frac{e^{\kappa(1+t)^{1-\nu}}}{1+t}}=1.
\end{equation}
This implies that there exits a constant $M'>0$ independent of $t$ such that%
\begin{equation}
r(t)\leq\frac{M'}{1+t},~\forall t\geq0.\label{eq6}%
\end{equation}
Finally, by combining the estimations (\ref{eq1})-(\ref{eq6}) we obtain the required result (\ref{rat}).
\end{proof}

\end{document}